\documentclass[12pt,leqno]{article}
\usepackage{amsfonts}
\usepackage{amsmath}

\columnwidth\textwidth
\newtheorem{theorem}{Theorem}[section]

\newtheorem{corollary}[theorem]{Corollary}

\newtheorem{lemma}[theorem]{Lemma}

\newtheorem{proposition}[theorem]{Proposition}
\newtheorem{remark}[theorem]{Remark}

\newenvironment{proof}[1][Proof]{\noindent\textbf{#1.} }{\ \rule{0.5em}{0.5em}}

    \newcommand{\cb}{{\mathcal B}}   
\newcommand{\cm}{{\mathcal M}}   
\numberwithin{equation}{section}

\begin{document}

\begin{center}
{\large Positive solutions to semilinear Dirichlet problems with general boundary data}\\[3mm]
Lucian Beznea\footnote{%
Simion Stoilow Institute of Mathematics of the Romanian Academy, 
%Research unit No. 2,  
P.O. Box \mbox{1-764,} RO-014700 Bucharest, Romania, and 
National University of Science and Technology Politehnica Bucharest, CAMPUS Institute. {E-mail}: {\
lucian.beznea@imar.ro}} and {Alexandra Teodor}\footnote{%
National University of Science and Technology Politehnica Bucharest and Simion Stoilow Institute of
Mathematics of the Romanian Academy. {E-mail}:
alexandravictoriateodor@gmail.com}
\end{center}

\vspace{5mm}
\begin{abstract}
	We give a probabilistic representation of the solution to a semilinear
	elliptic Dirichlet problem with general (discontinuous) boundary data. 
	The boundary behaviour of the solution is in the sense of the
	controlled convergence initiated by A. Cornea.
	Uniqueness results for the solution are also provided.
\end{abstract}


\noindent
{\bf Mathematics Subject Classification (2010):} 
%60H15,     % Stochastic partial differential equations [See also 35R60]
%60H10,     % Stochastic ordinary differential equations [See also 34F05]
35J65,       % Nonlinear boundary value problems for linear elliptic equations
60J45,  	% Probabilistic potential theory
60J35,    	% Transition functions, generators and resolvents
%60J40,  	% Right processes
60J57,  	% Multiplicative functionals
%31C25,      % Dirichlet spaces
%47D07,  	% Markov semigroups and applications to diffusion processes 
%60J25,      % Continuous-time Markov processes on general state spaces
%60J60,      % Diffusion processes
%37C40 (primary), 37A30, 37L40, 60J35, 60J25, 60J60, 31C25, 37C40,  82B10 (secondary)
%47D03  	%Groups and semigroups of linear operators 
%47A35  	%Ergodic theory
%37A30   %Ergodic theorems, spectral theory, Markov operators
%37C40   %Smooth ergodic theory, invariant measures
%37L40   %Invariant measures (Infinite-dimensional dissipative dynamical systems)
%60J55  	Local time and additive functionals
%60J25  	Continuous-time Markov processes on general state spaces
%60J35: Transition functions, generators and resolvents
%82B10  	Quantum equilibrium statistical mechanics (general)
31C05, 	%Harmonic, subharmonic, superharmonic functions
% 60J45,     %Probabilistic potential theory
% 60J60: Diffusion processes
60J65.  % Brownian motion

%35R60,       % Partial differential equations with randomness, stochastic partial differential equations

\vspace{5mm}

\section{Introduction} 

We consider the following semilinear elliptic equation with Dirichlet
boundary conditions:
\begin{equation} \label{prob initiala}
\left\{ 
\begin{array}{c}
\frac{1}{2}\Delta u-F\left( x,u\right) =0\text{ } \\[4.1mm] 
\hspace*{-2mm} \underset{D\ni x\rightarrow y}{\lim }u\left( x\right) =\phi \left( y\right) 
\text{ }%
\end{array}%
\right. 
\begin{array}{c}
\text{for }x\in D, \\[4mm]
\text{ for }y\in \partial D,%
\end{array}
\end{equation}%
where $D$ is  a bounded, regular domain in $\mathbb{R}^{d},$ $d\geq
3,$ the boundary condition $\phi $ is a non-negative, bounded and Borel
measurable real-valued function defined on the boundary $\partial D$ of $D$, 
and $F$ is a
real-valued Borel measurable function on $D\times \left( 0,b\right) $ for
some $b\in (0,\infty ]$ such that $F\left( x,\cdot \right) $ is continuous
on $(0,b)$ for every $x\in D$ and $0\leq F\left( x,u\right) \leq U\left(
x\right) u$ for every $\left( x,u\right) \in D\times \left( 0,b\right) ,$
where $U$ is a positive Green-tight function on $D$.

A function $u\in C\left( D\right) $ solving (\ref{prob initiala}) is
called the \textit{classical solution }of the nonlinear Dirichlet problem (%
\ref{prob initiala}) with boundary data $\phi ,$ associated with the
operator $u\longmapsto \frac{1}{2}\Delta u-F\left( \cdot ,u\right)$. 
Here a positive solution means a solution that is strictly positive on $D$.

If $\phi $ is a non-negative and continuous real-valued function on $%
\partial D,$ then (\ref{prob initiala}) is a special case of the problem
studied by Chen, Williams, and Zhao in \cite{Chen93}. 
Under  additional restrictions on $\phi $, they proved the existence of positive continuous
solutions to the problem (\ref{prob initiala}) in the weak sense (Theorem
1.1 from \cite{Chen93}), using an implicit probabilistic representation
together with Schauder's fixed point theorem and compactness criteria on spaces of continuous functions.
The existence of positive solutions of singular nonlinear elliptic equations (of the type $(\ref{prob initiala})$) with Dirichlet boundary conditions
was also studied in \cite{Hir}.

If $\phi $ is discontinuous on 
$\partial D$ then the problem (\ref{prob initiala}) has no classical
solution (see Proposition \ref{prop2.2} below), 
therefore we need a more general kind of solution for the above
nonlinear Dirichlet problem if the boundary data $\phi $ is not continuous.
%As in the linear case %\cite{Be11} 

The purpose of this paper is to present a method of solving equation $(\ref{prob initiala})$ for discontinuous boundary data.
Instead of the pointwise convergence we use A. Cornea's  controlled convergence   to the boundary data  (cf.  \cite{Cornea95} and \cite{Cornea98}).
It turns out that this type of convergence provides a way to describe the boundary behaviour of the solution to the boundary value  problems for general (not necessarily
continuous) boundary data and it was already used 
in the linear case
 for the Dirichlet problem on an Euclidean domain  (cf. \cite{Be11}) but also for the Dirichlet problem associated with the Gross-Laplace operator on an abstract Wiener space (in \cite{BeCoRo11}), and 
for the Neumann problem  on an Euclidean ball (see \cite{BePaPaControl}); see also the Remark \ref{rem2.1} below.

Our strategy  is to  modify the procedure of \cite{Chen93} for solving $(\ref{prob initiala})$,
since we have to work with spaces of discontinuous functions, in particular, the above mentioned compactness criteria are not more suitable. 
However, the imposed  additional hypothesis on the nonlinear term $F$ permits  us to use  Banach fixed point theorem. 
As a byproduct we prove an uniqueness result and a probabilistic representation of the solution to $(\ref{prob initiala})$, 
an approximation with stochastic terms, which might be considered  
an analogue of the stochastic solution to the linear Dirichlet problem. 
Note that in \cite{AAC} the authors mention an "implicit" probabilistic representation of the solution,
 in  the case when $F(\cdot ,u)=u^{p}$,  and $\phi$ a non-negative continuous  function on the boundary.

The structure of the paper is the following.
The controlled convergence is introduced at the begining of Section \ref{section2}. 
We expose  then the linear Dirichlet problem with general boundary data (Theorem \ref{thm2.5}), based on the controlled convergence, 
 improving essentially the result from \cite{Be11} which shows that that the stochastic solution solves the problem in this case;  we put its proof the Appendix of the paper.
 In Theorem \ref{unicitate}  we show the uniqueness of  the solution to the Dirichlet problem associated with the operator 
 $\frac{1}{2} \Delta +q$, it is  a generalization to discontinuous boundary data of a result from \cite{ChungZhao95} (Theorem 3.21).
The main result of the paper (Theorem \ref{thm3.1}), the existence of the solution to equation $(\ref{prob initiala})$, 
 and its proof are  presented in Section \ref{section3}. 
 The  probabilistic representation of the solution  is stated in Remark \ref{rem3.3}.
 We complete  the section  by proving  (in Theorem \ref{thm3.3}) the uniqueness of the solution to equation $(\ref{prob initiala})$.

The authors

\section{Solutions to the Dirichlet problem based on controlled convergence} \label{section2}

\noindent
{\bf Controlled convergence.} 
Let $f:\partial D\longrightarrow \overline{\mathbb{R}}$ and $%
h,k:D\longrightarrow {\mathbb{R}},$ $k\geq 0.$  
We say that \textit{$h$ converges to $f$
controlled by $k$} (we write $h\overset{k}{\longrightarrow }f$) provided that for every $A\subset D$ and $y\in \partial D\cap \overline{A}$ the following conditions hold:

\begin{enumerate}
\item[$\left( \ast \right) $] If $\underset{A\ni x\rightarrow y}{\lim \sup }$
$k\left( x\right) <\infty $ then $f\left( y\right) \in \mathbb{R}$ and $f\left(
y\right) =\underset{A\ni x\rightarrow y}{\lim }h\left( x\right) $

\item[$\left( \ast \ast \right) $] If $\underset{A\ni x\rightarrow y}{\lim }$
$k\left( x\right) =\infty $ then $\underset{A\ni x\rightarrow y}{\lim }\frac{%
h\left( x\right) }{1+k\left( x\right) }=0.$
\end{enumerate}
The function $k$ is called a {\it  control function} for $f$.

\begin{remark} \label{rem2.1} % Remark 2.1
\begin{enumerate}
\item[(i)] This type of convergence was introduced by A. Cornea in \cite{Cornea95}
and \cite{Cornea98}. A first motivation was to solve the (linear) Dirichlet
problem with general (discontinuous) boundary data. 
The monograph \cite{Lukes}
presents the controlled convergence with application to the
Perron-Wiener-Brelot solution of a Dirichlet problem in the frame of a
harmonic space.

\item[(ii)] In \cite{BeCoRo11} the controlled convergence was used to state
and solve the Dirichlet problem for the Gross-Laplace operator on an
abstract Wiener space, for general (not necessarily continuous) boundary
data. 
Note that a main difficulty in proving the existence of the solution
is the construction of a convenient control function.

\item[(iii)] In \cite{BePaPafaracontrol} and \cite{BePaPaControl} the
controlled convergence was used to solve the Neumann problem on a ball with
discontinuous boundary data.

\item[(iv)] One can  see that  if $h_i$ converges controlled by $k_i$ to the real-valued function $f_i$, $i=1,2,$ and $\alpha \in  \mathbb{R}$, then
$h_1+ \alpha h_2$ converges controlled by $k_1+k_2$ to $f_1+ \alpha f_2$.

\end{enumerate}
\end{remark}

We denote by $b\mathcal{B}_+(A) $ the set of all bounded,
positive and Borel measurable functions on a set $A$, by $C(A)$ the set of all continuous functions on $A$, by $C_{b}(A)$ the set of all bounded and continuous functions on $A$, and by $\left\vert \left\vert
\cdot \right\vert \right\vert _{\infty }$ the supremum norm of a
real-valued function over its domain of definition, in particular $%
\left\vert \left\vert \phi \right\vert \right\vert _{\infty }=\underset{x\in
	\partial D}{\sup }\left\vert \phi \right\vert ,$ provided that $\phi $ is
defined on the boundary $\partial D$ of $D$. 
%We denote by $LA,\lambda$ the class of Borel measurable functions on a set $A$ such that $||f||_{p}:=\int_{A}}$

\begin{proposition}\label{prop2.2} % Proposition 2.2
\label{remark 1.1}If the nonlinear Dirichlet problem (\ref{prob initiala})
has a classical solution $u\in C\left( D\right) $ then the boundary data 
$\phi $ is continuous on $\partial D$.
\end{proposition}

\begin{proof}
See the proof of Remark 1.1 from  \cite{Be11}.
\end{proof}

\medskip

According to Proposition \ref{remark 1.1} if the boundary data $\phi $ is
discontinuous then the problem (\ref{prob initiala}) has no classical
solution. Below we introduce a more general solution to our problem (\ref%
{prob initiala}).

A function $u\in C\left( D\right) $ is called a \textit{ weak solution} to the nonlinear Dirichlet problem with boundary data 
$\phi \in b\mathcal{B}_+\left( \partial D\right) $ associated with the operator $\frac{1}{2}%
\Delta u-F\left( \cdot ,u\right),$ 
provided that there exists a control
function $k$ which is superharmonic in $D$  such that 
\begin{equation*}
\left\{ 
\begin{array}{l}
\frac{1}{2}\Delta u-F\left( \cdot ,u\right) =0\text{ on }D,~\text{in the
weak sense} \\ [2mm]
u\text{ converges to }\phi \text{ controlled by }k.%
\end{array}%
\right. 
\end{equation*}

\subsection{The linear case}
A solution to the {\it classical Dirichlet problem on $D$ with boundary data} $f : \partial D \longrightarrow \mathbb{R}$ 
is a harmonic function $h : D \longrightarrow \mathbb{R}$ such that
$\mathop{\lim}\limits_{\mathop{x \to y}\limits_{x \in D}} h(x) = f(y) \, \mbox{ for all }  y \in \partial D.$ \\

%\noindent
%{\bf The stochastic solution to the Dirichlet problem.}
For $x\in \mathbb{R}$  let $
\left( X\left( t\right) ,t\geq 0\right) $ under the probability $\mathbb{P}^x$ 
be the $d$-dimensional Brownian motion starting from $x$, 
 denote by $\mathbb{E}^{x}$ the
expectation under $\mathbb{P}^{x}$, 
and let $\tau _{D}=\inf \{t>0:$ $X\left(
t\right) \notin D\}$ be the hitting time of $D^{c}$. 

If $ f: \partial D \longrightarrow \mathbb{R}$ is a  bounded below Borel measurable and $x\in D$ define
$$
H_D f(x) := \mathbb{E}^x f (X({\tau_{D})}) .
$$
By  Theorem 3.7 from \cite{Port-Stone}, page 106,  
it follows that $H_D f$ is a (real-valued) harmonic functions on $D$ provided that $H_D f$ is not equal $+\infty$ everywhere on $D$.
In addition, if $f\in C(\partial D)$ then $H_D f$ is the solution to the classical Dirichlet problem on $D$ with boundary data $f$.
Therefor  $H_D f$ is called {\it the stochastic solution to the Dirichlet problem}.

%By Lemma \ref{lem3.1} (i) we have  $H_D {1}(x) = 1$ and   $H_D f(x)$ is the expected value of $f$ at the points where the Brownian motion starting at $x$, exits from  $D$.

\noindent
{\bf The Dirichlet problem based on controlled convergence}. 
A function  $f : \partial D \longrightarrow  \overline{\mathbb{R}}$
is called {\it resolutive} provided that there exists a harmonic function $ h$  on $D$ which converges to $f$  controlled by a real-valued, non-negative superharmonic function $k$.  
The unique function $h$ (see Corollary \ref{cor2.4} below) is called the {\it solution on $D$ to the Dirichlet problem with boundary data} $f$.

\begin{remark} \label{rem4.1}  % Remark 2.3
$(i)$ According to \cite{Cornea95} and \cite{Cornea98}, and as mentioned in the Introduction, 
the controlled convergence offers a method for setting and solving 
the Dirichlet problem for general open sets and general boundary data. 
The above function $f$ should be interpreted as being the boundary data of the harmonic function $h$.

$(ii)$ A harmonic function $h$ on $D$ is the solution to the classical  Dirichlet problem with boundary data $f$ 
if and only if $  h$ converges to $f$ controlled by a bounded function $k$.
\end{remark}

The next corollary is a version of Corollary 4.3 from \cite{Be11}.

\begin{corollary} \label{cor2.4} % Corollary 2.4
If the  Dirichlet problem has a solution then it is unique. 
In particular, if $u$ is  a harmonic function on $D$ which converges controlled by $k$ to zero then $u=0$ on $D$.
\end{corollary}

\begin{proof}
We can argue as in the proof of Corollary 4.3 from \cite{Be11}. 
The only additional argument we need is the fact that in our case here we may assume that the control functions are continuous, according to Proposition 2.2 from \cite{Cornea98}.
\end{proof}

\vspace{2mm}

The next  result  shows that the stochastic solution solves the Dirichlet problem with general boundary data.
It is an improvement of the main result from \cite{Be11}, Theorem 4.8;
for the relation with the resolutivity for the method of Perron-Wiener-Brelot see Corollary 2.13 from \cite{Cornea98}.
Some arguments in the proof are like in the proofs of  Theorem 4.8 from \cite{Be11} and
Theorem 5.3 from \cite{BeCoRo11}. 
For the reader convenience we present in the Appendix the proof of the next theorem.

\begin{theorem} \label{thm2.5} %Theorem 2.5
Let $D \subset \mathbb{R}^d$ be a bounded domain such that the classical Dirichlet problem has a solution.
Let $f \in \cb_+(\partial D)$ and assume that $H_D f$ is not equal $+\infty$ everywhere on $D$.
%$\lambda$ a finite measure on $D$,  $\sigma := \lambda \circ H_D$, such that the classical Dirichlet problem has a solution, and let $f \in L^1(\partial D, \sigma)$.
Then $H_D f$
%the stochastic solution to the Dirichlet problem, 
is the unique solution to the 
Dirichlet problem with boundary data $f$.
More precisely, there exists  $g \in \mathcal{B}_+(\partial D)$ such that the function 
$k := H_D g$ is real-valued 
and  $H_D f$  converges to $f$ controlled by $k$. %so,  $h $ is the unique solution of the Dirichlet problem with boundary data $f.$
\end{theorem}

%\noindent
%{\bf An uniqueness result.}

Let $g:\mathbb{R}^d \longrightarrow \mathbb{R}\cup \{\infty \}$ be "the Green function" on $\mathbb{R}^d$, 
\[
g\left( u\right) =\left\{ 
\begin{array}{c}
\left\vert u\right\vert ^{d-2}, \\[2mm] 
\ln \frac{1}{\left\vert u\right\vert }, \\[2mm] 
\hspace*{-4mm} \left\vert u\right\vert ,%
\end{array}%
\right. 
\begin{array}{c}
\text{if }d\geq 3 \\[2mm]
\text{if }d=2\\[2mm]
\text{if }d=1. %
\end{array}%
\]%

\noindent
{\bf The Kato class.} 
Following \cite{ChungZhao95} we define the Kato class $J$ of the function $g$: 
the set of all real-valued Borel measurable functions $q$ defined on $\mathbb{R}^{d}$ such that
%$%q\in J$ iff%
\[
%q\in J \ \mbox{ if and only if }  
\  \lim_{\alpha \searrow 0}\left[ \sup_{x\in \mathbb{R}^{d}}\int_{\left\vert
y-x\right\vert \leq \alpha }\left\vert g\left( y-x\right) q\left( y\right)
\right\vert dy\right] =0.
\]
If the function $q$ is only defined on the domain $D$,  
then we extend it to $\mathbb{R}^{d}$ by setting it to vanish on the complement of $D$.\\
%for any real-valued Borel measurable function $q$ defined on $\mathbb{R}^{d}$.

\noindent
\textbf{Green-tight functions}. 
A function $w:D\longrightarrow \mathbb{R}$ is said to be\textit{\ Green-tight }%
on a bounded domain $D$ of $\mathbb{R}^{d}$ 
provided that it is Borel
measurable and such that 
\begin{equation*}
	\lim_{\substack{ \lambda \left( A\right) \rightarrow 0,  \\ A\subset D}}%
	\{\sup_{x\in D}\int_{A}\frac{|w\left( y\right) |}{|x-y|^{d-2}}dy\}=0,
\end{equation*}%
where $\lambda $ denotes the Lebesgue measure on $\mathbb{R}^{d}$. If $w$ is
Green-tight on $D$ then it satisfies 
\begin{equation*}
	\left\vert \left\vert w\right\vert \right\vert _{D}:=\sup_{x\in D}\int_{D}%
	\frac{\left\vert w\left( y\right) \right\vert }{\left\vert x-y\right\vert
		^{d-2}}dy<\infty .
\end{equation*}

Let $G$ be the Green function of the operator $\frac{1}{2}\Delta $ on $D.$
Then by (1.5) and (1.6) from \cite{Chen93}, 
\begin{eqnarray*}
	G\left( x,y\right) &=&0\text{ for all }x\in \partial D,\text{ }y\in D\text{
		and} \\
	0 &\leq &G\left( x,y\right) \leq c\left\vert x-y\right\vert ^{2-d}\text{ for
		all }x,y\in D,  \notag
\end{eqnarray*}%
with $c:=$ $\Gamma \left( \frac{d}{2}-1\right) /2\pi ^{d/2}.$

As in (1.9) from \cite{Chen93} and Theorem 3.2 of \cite{ChungZhao95} for any Green-tight function $q$ on $D,$%
\begin{equation*}
	Gq ( x ) :=\int_{D}G\left( x,y\right) q(y) dy,\text{ }x\in \overline{D}, 
\end{equation*}%
defines a bounded and continuous function on $\overline{D}$ such that %that $Gq=0$ 
\begin{equation}
		\lim_{x\rightarrow z}Gq\left( x\right) =0  \label{Gq=0} \ \mbox{ for any  } z\in\partial D.
\end{equation}
Furthermore, by Proposition 2.10 from \cite{ChungZhao95} $Gq$
satisfies the equation 
\begin{equation}
	\frac{1}{2}\Delta \left( Gq\right) =-q\text{ on }D  \label{ecuatia cu G}
\end{equation}%
in the weak sense, that is for every test function $\psi \in C_{c}^{\infty
}\left( D\right) :=\{f\in C^{\infty }\left( D\right) :$ $f$ has compact
support in $D\} $ $Gq$ satisfies the equation 
\begin{equation*}
	\frac{1}{2}\int_{D}Gq(x)\Delta \psi \left( x\right) dx=-\int_{D}q\left(
	x\right) \psi \left( x\right) dx.
\end{equation*}
%Also,  by Theorem 4.3 (ii) from  \cite{ChungZhao95} for every $z\in \partial D$
%we have that 

%thus $Gq$ is continuous and bounded on $\overline{D}$.

\begin{remark} \label{rem2.6}  % Remark 2.6
According to \cite{Chen93},  a real-valued Borel measurable function $w$ on $D$ is Green-tight if and only if $1_{D}w\in J$.
\end{remark}

%\vspace{2mm}

\subsection{Uniqueness of the solution  for the operator $\frac{1}{2} \Delta +q$}. %Subsection 2.2

For $q\in J$ let $\{T_{t}\}_{t}$ be the Feynman-Kac semigroup on $D$ associated with the
multiplicative functional 
$ \left(  e^{ \int_{0}^{t} q ( X( s) ) ds } \right)_{t\geq 0}$,  
\[
T_{t}f\left( x\right) =\mathbb{E}^{x} (  e^{\int_{0}^{t}q( X( s) )  ds } f ( X( t) ); t<\tau _{D}), \ x \in D.
\]%
Consider  further the potential operator of the semigroup $\{T_{t}\}_{t}$, 
\[
Vf\left( x\right) =\int_{0}^{\infty }T_{t}f\left( x\right)
dt=
\mathbb{E}^{x} \int_{0}^{\tau _{D}}
e^{\int_{0}^{t}q( X( s) )  ds} 
f( X(t) )dt 
\ \text{ for } f\in \mathcal{B}_{+} (D) .
\]

Let $q\in J$ and $\phi \in b\cb_+ ( \partial D)$.  
A function $u\in C\left( D\right) $ is called 
 \textit{weak solution }%
to the Dirichlet problem associated with the operator 
$\frac{1}{2} \Delta +q$,  
with boundary data $\phi $, provided that there exists a 
superharmonic control function $k:D\longrightarrow 
\mathbb{R}_{+}$ such that
\[
\left\{ 
\begin{array}{c}
\frac{1}{2}\Delta u+qu=0\text{ on }D\text{ in the weak sense} \\[2mm]
\hspace*{-8mm} u~\text{converges to }\phi \text{ controlled by }k.%
\end{array}%
\right. 
\]

\begin{theorem} \label{unicitate} % Theorem 2.7
%Let $D$ be an open and bounded subset of $\mathbb{R}^{d}$, $%
Let $q\in J$ and  $\phi \in b\cb_+( \partial D )$.
%and suppose 
%that $D$ is regular and  that $V1\in L^{\infty } ( D )$. 
If  the linear Dirichlet problem associated with the operator 
$\frac{1}{2} \Delta +q$,  
with boundary data $\phi $,  has a  weak solution in $C_{b}\left( D\right) $ then it is unique.
\end{theorem}

\begin{proof}
Let $u_{1},$ $u_{2}\in C_{b}\left( D\right) $ be two weak solutions
to the linear Dirichlet problem associated with the operator 
$\frac{1}{2} \Delta +q$,
with $k_{1}$ and, respectively, $k_{2},$ their
control functions, that is,  
$u_{1}\overset{k_{1}}{\longrightarrow }\phi $ 
and $u_{2}\overset{k_{2}}{\longrightarrow }\phi ,$ where $%
k_{1},k_{2}:D\rightarrow \mathbb{R}_{+}$ are two positive superharmonic functions. Let $v:=u_{1}-u_{2}$. 
By assertion $(iv)$ of Remark \ref{rem2.1} 
we have
$ \frac{1}{2}\Delta v+qv = 0$ in the weak sense on $D$ 
and 
$v\overset{k}{\longrightarrow }0,$
with $k=k_{1}+k_{2}.$ 
We use now some arguments from the proof of Theorem 3.21 from  \cite{ChungZhao95}. 
Let $f:=v-G\left( qv\right)$.
Then  %solutions $u_{1}$ and $u_{2}$ are bounded in $D$, thus 
$v$ is 
bounded and since $q\in J$ we have that $1_{D}qv\in J\cap L^1(D)$, using also Proposition 3.1 from  \cite{ChungZhao95}.
By
Theorem 3.2 from \cite{ChungZhao95} we conclude that  $G\left( \left\vert qv\right\vert \right) 
$ is bounded in $D$ and so it belongs to $L^{1}\left( D\right) $. 
Therefore, by Proposition 2.10 from  \cite{ChungZhao95} we have 
\[
\frac{1}{2}\Delta G\left( qv\right) =-qv\text{ on }D\text{ in the weak sense,%
}
\]%
hence
$\Delta f=\Delta v-\Delta G\left( qv\right) =-2qv+2qv=0\text{ on }D$  in
the weak sense.
Again by Theorem 3.2 from  \cite{ChungZhao95} we get that $G\left( qv\right) \in C_{0}\left(
D\right) $, therefore $f\in C \left( D\right)$.  
Thus by Weyl's lemma  (see e.g. \cite{Daco}, page 118) $f$
is harmonic on $D$. 
We have also 
$\lim_{x\rightarrow z}G\left( qv\right) \left( x\right) =0\text{ for all }%
z\in \partial D,$
and since $v=u_{1}-u_{2}\overset{k_{1}+k_{2}}{\longrightarrow }0,$ we get 
that $f\overset{k_{1}+k_{2}}{\longrightarrow }0.$
Thus $f$ is a solution to the Dirichlet problem%
\begin{equation}
\left\{ 
\begin{array}{c}
\hspace*{-46mm} \Delta w=0\text{ on }D \\[3mm] 
w\text{ converges to }0\text{ controlled by } k_1+k_2. %
\end{array}%
\right.   \label{Dirichlet liniara 1}
\end{equation}%
By Corollary \ref{cor2.4} it follows that
$f=0$ on  $D$,
hence $v=G\left( qv\right) $ on $D.$ 
Further, we argue as in the last part of the  proof of Theorem 3.21 from  \cite{ChungZhao95}.
%Since $D$ is bounded, $%
By
Theorem 3.2 from \cite{ChungZhao95} we have that $G|q|$ is bounded in $D$, hence 
$G\left\vert qv\right\vert \leq \left\vert \left\vert v\right\vert \right\vert
_{\infty }\left\vert \left\vert G|q|\right\vert \right\vert _{\infty }<\infty $ 
and therefore $V\left( \left\vert q\right\vert G\left\vert qv\right\vert
\right) <\infty $ by Theorem 3.18 from \cite{ChungZhao95}. 
Note that $V$ is a bounded kernel 
because 
$V1(x)\leq \mathbb{E}^x [\tau_D]\leq \frac{R^2}{d}$, 
where $B$ is a ball of radius $R$ centered at the origin, containing $D$; 
see $(\ref{eq3.6})$ below for more details on the proof of the last inequality.
Thus, by (45)  from \cite{ChungZhao95} %with $\lambda =0$ and $f=qv,$ 
we obtain 
$V\left( qv\right) =G\left( qv\right) +V\left( qG\left( qv\right) \right) $, 
hence 
$V\left( qv\right) =v+V\left( qv\right)$ on  $D$.
Since $V\left( qv\right) <\infty ,$ we get $v=0$ on $D,$ therefore $%
u_{1}=u_{2}$ on $D$.
\end{proof}

\section{The nonlinear Dirichlet problem}\label{section3}. % Section 3

As in \cite{Chen93}, (3.8)-(3.10), we fix a positive Green-tight function $U$
on $D$ such that $0\leq F\left( x,u\right) \leq U\left( x\right) u$ for all $%
(x,u)\in D\times \left( 0,b\right) $ for some $b\in (0,\infty ]$.
Assume that $\| \phi \|_{\infty }< b$ and
let 
\begin{equation*}
\gamma _{0}:=\inf \{\phi \left( x\right) :x\in \partial D\},\text{ }\beta
:=c\left\vert \left\vert U\right\vert \right\vert _{D},
\end{equation*}%
and 
\begin{equation}
\Lambda :=\{u\in b\mathcal{B} ( D) : m:=e^{-\beta }\gamma _{0}\leq
u\leq \left\vert \left\vert \phi \right\vert \right\vert _{\infty }=:%
\widetilde{m}\text{ on }D\}.  \label{lambda}
\end{equation}%
%We remark that 
We endow $\Lambda $ with the metric induced by the supremum norm and clearly we obtain a complete metric space. 
%with respect to the supremum norm $\left\vert \left\vert \cdot \right\vert \right\vert _{\infty}.$

If $\gamma _{0}>0$, since $m=e^{-\beta }\gamma _{0}>0$, we may define for every $x\in D$ the function 
$H_{x}:[m,\widetilde{m}]\longrightarrow \lbrack 0,\infty )$ as
%\begin{equation}
$$
H_{x}\left( y\right) :=\frac{F\left( x,y\right) }{y}, \ y\in [m,\widetilde{m}].
$$

We can state now the main result of this paper. 
Consider a ball  $B$ of radius $R$  centered at the origin, containing $D$.

\begin{theorem} \label{thm3.1} % Theorem 3.1
\label{teorema} 
Let $\phi >0$ be a bounded and Borel measurable function on $%
\partial D$ such that $\gamma _{0}>0$. 
Assume that for every $x\in D$ the
function $H_{x}$ is Lipschitz continuous on $[m,\widetilde{m}]$ with the
constant $C$ that does not depend on $x$ and let $\phi $ be such that 
\begin{equation}
\left\vert \left\vert \phi \right\vert \right\vert _{\infty }<\frac{d}{R^{2}C%
}.  \label{conditie2}
\end{equation}%
Then the nonlinear Dirichlet problem with boundary
data $\phi $ associated with  the operator $u\longmapsto \frac{1}{2}\Delta
u-F\left( \cdot ,u\right) $ has a weak solution $u\in C\left(
D\right)$,  that is, %
\begin{equation}
\left\{ 
\begin{array}{l}
\frac{1}{2}\Delta u-F\left( \cdot ,u\right) =0\text{ on }D \text{ in the
weak sense} \\[3mm]
u\text{ converges to }\phi \text{ controlled by }k,  
\end{array}%
\right.   \label{prob1}
\end{equation}
where the control function is  $k:= H_D g$ 
%$k:=\mathbb{E}^{\cdot }\left[ g\left( X\left( \tau _{D}\right) \right) \right] $ 
for some function $g\in $ $\mathcal{B}_+ ( D) $.
\end{theorem}

\begin{proof}
We use some arguments from the proof of Theorem 1.1 in \cite{Chen93}. 
 For any real-valued
Borel measurable function $w$ defined on $D$ such that $\left\vert
\left\vert w\right\vert \right\vert _{D}<\infty $ and for every $x\in D,$ $%
\mathbb{P}^{x}$-a.s., we consider the following stopped Feynman-Kac functional which
is well defined, positive and finite for all $t\geq 0$ 
\begin{equation*}
e_{w}\left( t\right) :=\exp (\int_{0}^{t\wedge \tau _{D}}w\left( X\left(
s\right) \right) ds)\text{.}
\end{equation*}%
For any $u\in \Lambda $ and $x\in D$ the function 
\begin{equation*}
q_{u}\left( x\right) :=\frac{-F\left( x,u\left( x\right) \right) }{u\left(
x\right) }\text{ }
\end{equation*}%
is well defined and 
\begin{equation*}
-U\left( x\right) \leq q_{u}\left( x\right) \leq 0\text{ and }\left\vert
q_{u}\left( x\right) \right\vert \leq U\left( x\right).
\end{equation*}%
Thus 
%\begin{equation*}
$\left\vert \left\vert q_{u}\right\vert \right\vert _{D}\leq \left\vert \left\vert U\right\vert \right\vert _{D}<\infty$ and $q_u$ is Green-tight on $D$ 
or equivalently (cf. Remark \ref{rem2.6}), $1_Dq_u$ belongs to $J$.
%\end{equation*}%

As in (3.15) and (3.16) from  \cite{Chen93} 
\begin{equation*}
\mathbb{E}^{x} \int_{0}^{\tau _{D}}U\left( X\left( t\right) \right) dt \leq
c\left\vert \left\vert U\right\vert \right\vert _{D}=\beta
\end{equation*}%
and 
\begin{equation}
e^{-\beta }\leq \mathbb{E}^{x}\left[ e_{-U}\left( \tau _{D}\right) \right]
\leq \mathbb{E}^{x}\left[ e_{q_{u}}\left( \tau _{D}\right) \right] \leq 1%
\text{ for each }x\in D.  \label{gogu1}
\end{equation}%

We define the operator $T$ on $\Lambda$ as  in (3.18) from \cite{Chen93}, 
\begin{equation*}
Tu\left( x\right) :=\mathbb{E}^{x}\left[ e_{q_{u}}\left( \tau _{D}\right) \phi \left(
X\left( \tau _{D}\right) \right) \right] \text{ for all }  u\in \Lambda  \mbox{ and } x\in D.
\end{equation*}%
By (\ref{gogu1}), we have that 
\begin{equation*}
e^{-\beta }\gamma _{0}\leq \mathbb{E}^{x}\left[ e_{-U}\left( \tau
_{D}\right) \phi \left( X\left( \tau _{D}\right) \right) \right] \leq
Tu\left( x\right) \leq \left\vert \left\vert \phi \right\vert \right\vert
_{\infty }\mathbb{E}^{x}\left[ e_{q_{u}}\left( \tau _{D}\right) \right] \leq
\left\vert \left\vert \phi \right\vert \right\vert _{\infty },
\end{equation*}%
thus $Tu$ is well defined, finite and bounded on $D$ and in particular 
%\begin{equation*}
$T\Lambda \subset \Lambda .$
%\end{equation*}%

We show that $T$ is a contraction map on $\Lambda $ with respect to the
supremum norm. % $\left\vert \left\vert \cdot \right\vert \right\vert _{\infty}.$ 
For any $x\in D$ and $u,$ $v\in \Lambda $ we have that%
\begin{eqnarray}
\left\vert Tu\left( x\right) -Tv\left( x\right) \right\vert &\mathbf{=}%
&\left\vert \mathbb{E}^{x}\left[ e_{q_{u}}\left( \tau _{D}\right) \phi
\left( X\left( \tau _{D}\right) \right) \right] -\mathbb{E}^{x}\left[
e_{q_{v}}\left( \tau _{D}\right) \phi \left( X\left( \tau _{D}\right)
\right) \right] \right\vert  \notag \\
&\leq &\left\vert \left\vert \phi \right\vert \right\vert _{\infty }\mathbb{E%
}^{x}\left[ \left\vert e_{q_{u}}\left( \tau _{D}\right) -e_{q_{v}}\left(
\tau _{D}\right) \right\vert \right]  \notag \\
&=&\left\vert \left\vert \phi \right\vert \right\vert _{\infty }\mathbb{E}%
^{x}[|e^{-\underset{0}{\overset{\tau _{D}}{\int }}\frac{F\left(
X(s),u\left( X(s)\right) \right) }{u\left( X(s)\right) }ds}-e^{-\underset{%
0}{\overset{\tau _{D}}{\int }}\frac{F\left( X(s),v\left( X(s)\right)
\right) }{v\left( X(s)\right) }ds}|].  \notag
\end{eqnarray}%
From 
$\left\vert e^{-x}-e^{-y}\right\vert \leq \left\vert x-y\right\vert$
for all $x,$ $y\geq 0,$ we get
\begin{eqnarray*}
\left\vert Tu\left( x\right) -Tv\left( x\right) \right\vert &\leq
&\left\vert \left\vert \phi \right\vert \right\vert _{\infty }\mathbb{E}%
^{x} \left\vert \underset{0}{\overset{\tau _{D}}{\int }}\frac{F\left(
X(s),u\left( X(s)\right) \right) }{u\left( X(s)\right) }ds-\underset{0}{%
\overset{\tau _{D}}{\int }}\frac{F\left( X(s),v\left( X(s)\right) \right) 
}{v\left( X(s)\right) }ds\right\vert  \\
&=&\left\vert \left\vert \phi \right\vert \right\vert _{\infty }\mathbb{E}%
^{x}\underset{0}{\overset{\tau _{D}}{\int }}%
|H_{x}(u(X(s)))-H_{x}(v(X(s)))|ds.
\end{eqnarray*}%
Since, by hypothesis, $H_{x}$ is Lipshitz continuous with the constant $C$ 
for every $x\in D,$ we have that 
\begin{equation} \label{ineq2}
\left\vert Tu\left( x\right) -Tv\left( x\right) \right\vert 
\leq \left\vert \left\vert \phi \right\vert \right\vert _{\infty }
\mathbb{E}^{x}
\underset{0}{\overset{\tau _{D}}{\int }}C|u(X(s))-v(X(s))|ds\leq 
\end{equation}
$$
\left\vert \left\vert \phi \right\vert \right\vert _{\infty
}C\left\vert \left\vert u-v\right\vert \right\vert _{\infty }
\mathbb{E}^{x} \int_{0}^{\tau _{D}}ds 
=\left\vert \left\vert \phi \right\vert \right\vert _{\infty }C\left\vert
\left\vert u-v\right\vert \right\vert _{\infty }\mathbb{E}^{x}[\tau _{D}] .
$$
Let $\tau _{B}=\inf \{t>0$ $:$ $X(t) \notin B\}$ be the exit time of $B,$
then $\mathbb{E}^{x}[\tau _{D}]\leq \mathbb{E}^{x}[\tau _{B}]$ and
by  (7.4.2)
from \cite{Oksental}, we have that 
\begin{equation} \label{eq3.6}
\mathbb{E}^{x}[\tau _{B}]=\frac{1}{d}\left( R^{2}-\left\vert x\right\vert
^{2}\right) \leq \frac{R^{2}}{d}.
\end{equation}%
Then by (\ref{ineq2}) we have 
%\begin{equation}
$\left\vert Tu\left( x\right) -Tv\left( x\right) \right\vert \leq \left\vert
\left\vert \phi \right\vert \right\vert _{\infty }C\frac{R^{2}}{d}\left\vert
\left\vert u-v\right\vert \right\vert _{\infty }$ %\label{ineq3}
%\end{equation}%
and applying the supremum we get
%\begin{equation*}
$\left\vert \left\vert Tu-Tv\right\vert \right\vert _{\infty }\leq \left\vert
\left\vert \phi \right\vert \right\vert _{\infty }C\frac{R^{2}}{d}\left\vert
\left\vert u-v\right\vert \right\vert _{\infty }.$
%\end{equation*}%
From (\ref{conditie2}) we have that $0\leq \widetilde{C}:=$ $%
\left\vert \left\vert \phi \right\vert \right\vert _{\infty }C\frac{R^{2}}{d}%
<1$ and thus the operator $T$ is a contraction map over the complete metric
space $\Lambda $ (with respect to metric induced by the supremum norm).
Applying the
Banach fixed-point theorem, there exists a unique fixed-point $u_{0}\in
\Lambda $ such that $Tu_{0}=u_{0}.$ We need now the following lemma.

\begin{lemma} \label{Lemma}  % Lemma 3.2
Let $u\in \Lambda $ and suppose that $\phi >0$ is a bounded and Borel measurable function on $\partial D.$ 
Then $Tu$ is a continuous weak solution to 
%the nonlinear Dirichlet problem with boundary data $\phi $ associated  with  the
%operator $v\longmapsto \frac{1}{2}\Delta v-\frac{F\left( \cdot ,u\right) }{u} v$,  
%has a weak  solution $v\in C\left( D\right)$,   that is of 
the problem
\begin{equation}
\left\{ 
\begin{array}{c}
\frac{1}{2}\Delta v-\frac{F\left( \cdot ,u\right) }{u}v=0\text{ on }D \text{
in the weak sense} \\[2mm]
\hspace*{-15mm} v\text{ converges to }\phi \text{ controlled by }k,%
\end{array}%
\right.   \label{prob2}
\end{equation}
where the control function is $k:= H_D g$
for some  function $g\in $ $\mathcal{B}_+ ( D) $.
\end{lemma}

\noindent
\textbf{Proof of Lemma \ref{Lemma}.} 
We consider the linear Dirichlet problem on $D$ with
boundary data $\phi \in b\mathcal{B}_{+}\left( \partial D\right)$.
Let $h$ be
the stochastic solution to  this problem, $h=H_D \phi$.
%\begin{equation*} \text{ }h\left( x\right) :=\mathbb{E}^{x}\left[ \phi \left( X\left( \tau_{D}\right) \right) \right] ~\text{for all }x\in D. \end{equation*}%
We already noted that $h$ is a bounded harmonic function on $D$.
%It is well known that the function $h$ is harmonic and bounded on $D$ (cf.  \cite{Port-Stone}, Proposition 2.1, page 89). 
%Because $\phi $ is not necessarily continuous we have 
By Theorem \ref{thm2.5}  there exists a function $g\in $ 
$\mathcal{B}_+ ( D) $ such that $h$ converges to $\phi $ controlled
by $k$,  %on $D_{0}:=\left[ k<\infty \right] ,$ 
where $k:= H_D g$
%$k:=\mathbb{E}^{\cdot }% \left[ g\left( X\left( \tau _{D}\right) \right) \right] $ 
 is
a real-valued, positive superharmonic control function, 
%and $M:=\left[ k=\infty \right] $ is a polar set, 
that is 
\begin{equation}
h\overset{k}{\longrightarrow }\phi.  %\text{ on }D_{0}=\left[ k<\infty \right] .
\label{control}
\end{equation}%
As in (3.26) from \cite{Chen93}, by a straightforward calculation, Fubini's
theorem (in which the required absolute integrability is implied by $%
\left\vert \left\vert q_{u}\right\vert \right\vert _{D}<\infty ),$ the
strong Markov property, and the regularity of $D$, we have that for each $%
x\in D$ 
\begin{equation} 
	Tu\left( x\right) -h\left( x\right) =G\left( q_{u}Tu\right) \left( x\right).	
\label{h + G}
\end{equation} 
From the boundedness of $Tu$ and the fact that $U$ is Green-tight on $D$ and 
$\left\vert q_{u}\left( \cdot \right) \right\vert \leq U\left( \cdot \right) 
$ we have that $q_{u}Tu$ is Green-tight on $D$, thus $G(q_{u}Tu)$ is
continuous and bounded on $\overline{D}$ and it satisfies the equation (\ref{ecuatia cu
G}) with $q_{u}Tu$ instead of $q,$ that is 
\begin{equation}
\frac{1}{2}\Delta G(q_{u}Tu)=-q_{u}Tu\text{ on }D \mbox{ in weak sense}, 
\label{ec pt Green}
\end{equation}%
and from (\ref{Gq=0}) 
\begin{equation}
\lim_{D\ni x\rightarrow y}G(q_{u}Tu)\left( x\right) =0\text{ for every }y\in
\partial D.  \label{G egal cu 0}
\end{equation}%
Furthermore, we have that $\Delta h=0$ on $D$, thus by (\ref{ec pt Green}) $Tu$ %=h+G(q_{u}Tu)$
satisfies the equation 
\begin{equation*}
\frac{1}{2}\Delta v+q_{u}v=0~\text{ on }D  \mbox{ in the weak sense.}\label{ecuatia impartita}
\end{equation*}%
Hence by (\ref{control}), (\ref{h + G}), and (\ref{G egal
cu 0}) $Tu$ converges to $\phi $ controlled by $k$. 
Therefore $Tu$ is continuous and bounded on $D$ and it is a weak  solution to 
the problem (\ref{prob2}), so, the proof of Lemma \ref{Lemma} is complete.

We return now to the proof of Theorem \ref{teorema}. 
By %(\ref{prob2}) of
Lemma \ref{Lemma} it follows that the fixed point $u_{0}=Tu_{0}$ satisfies
the equation $\frac{1}{2}\Delta u_{0}+q_{u_{0}}u_{0}=0~$ on $D$ and $u_{0}$
converges thus to $\phi $ controlled by $k$, thus $%
u_{0}$ is a weak solution to  the problem (\ref{prob1}). 
\end{proof}

\vspace{2mm}

\begin{remark} \label{rem3.3} %Remark 3.3
The proof of Theorem \ref{thm3.1} allows to emphasize  the following  probabilistic representation 
of the solution to the nonlinear Dirichlet problem  $(\ref{prob initiala})$.
Let $v_0\in \Lambda$ and define recurrently 
$$
v_{n+1}:=\mathbb{E}^{\bf\cdot } [ e_{q_{v_n}}\left( \tau _{D}\right) \phi \left(
X\left( \tau _{D}\right) \right) ] \  \mbox{ for }  \  n\geq 0.
$$ 
Then the sequence $(v_n)_{n\geq 0}$ from $\Lambda$ converges uniformly to the solution to the  problem $(\ref{prob initiala})$
\end{remark}
 %\text{ for all }  u\in \Lambda  \mbox{ and } x\in D. \end{equation*}%

\noindent
{\bf Examples of functions $F$ satisfying the condition from Theorem \ref{thm3.1},} that is,
the Lipschitz constant of $H_x$ does not depend on $x$,  i.e., 
%\begin{equation*}
$\sup_{x\in D} \sup_{y\neq z}\frac{\left\vert H_{x}\left( y\right)
-H_{x}\left( z\right) \right\vert }{\left\vert y-z\right\vert }  <\infty$.

$(1)$ Assume that for every $x\in D,$ the function $H_{x}$ is continuous on $%
[m,\widetilde{m}]$ and differentiable on $(m,\widetilde{m}),$ such that 
%\begin{equation*}
$C:=\sup_{x\in D} \sup_{y\in (m,M)}|H_{x}^{\prime }\left( y\right) | <\infty.$
%\end{equation*}%
Then clearly, % from the mean-value theorem 
for every $x\in D,$ $H_{x}$ is Lipschitz continuous on $[m,\widetilde{m}]$ with the constant $C.$

$(2)$ 
Suppose that $F$ has an  extension of a class $C^1$ to $\overline{D}\times \lbrack m,\widetilde{m}]$.
Then $F$ satisfies the required condition.

\vspace{2mm}

We close this section with  a result of uniqueness of the solution to equation $(\ref{prob initiala})$.

\begin{theorem} \label{thm3.3} % Theorem 3.4
If the nonlinear Dirichlet problem associated with 
the operator $v\longmapsto \frac{1}{2}\Delta v -F\left( \cdot ,v\right) $, 
with boundary data $\phi \in b\mathcal{B}_+\left( \partial D\right) $, 
 has a  weak solution in $ \Lambda$
 %\in C_{b}\left( D\right) $ 
 then it is unique.
\end{theorem}

\begin{proof}
Let $u_{1}, u_{2}\in \Lambda$ be two weak solutions  to the nonlinear Dirichlet problem associated
with  the operator $v \longmapsto \frac{1}{2}\Delta v-F\left( \cdot ,v \right)$,
with boundary data $\phi $.
Therefore $u_{1}$ and $u_{2}$ are also weak solutions to the linear Dirichlet
problems associated with the operator $v\longmapsto \frac{1}{2}\Delta v-\frac{%
F\left( \cdot ,u_{1}\right) }{u_1}v$ and respectively $v\longmapsto \frac{1}{2}\Delta v-\frac{%
F\left( \cdot ,u_{2}\right) }{u_2}v$. 
Then $q_{u_{1}}=-\frac{F\left( \cdot ,u_{1}\right) }{u_{1}}$ and $q_{u_{2}}=-\frac{F\left( \cdot ,u_{2}\right) }{u_{2}}$ 
%it follows that $q_{u_{1}}$ and $q_{u_{2}}$ 
are Green-tight functions on $D$ and by Theorem \ref{unicitate} 
we have that $u_{1}$ and $u_{2}$ are the unique weak solutions to the linear
Dirichlet problems associated with the operator $\frac{1}{2} \Delta
+q_{u_{1}}$ and  respectively $\frac{1}{2} \Delta
+q_{u_{2}}$. 
By Lemma \ref{Lemma} we have that $u_{1}=Tu_{1}$ and $u_{2}=Tu_{2}$, thus $u_{1}$ and $u_{2}$ are two fixed-points of the contraction operator $T$ on $\Lambda$ (cf. the proof of Theorem \ref{thm3.1})  and  we conclude that $u_{1}= u_{2}$. 
\end{proof}

%\vspace{2mm}

%\noindent
%{\bf Acknowledgments.} 
%This work was supported by a grant of the Ministry of Research, Innovation and Digitalization, CNCS - UEFISCDI,
%project number PN-III-P4-PCE-2021-0921, within PNCDI III.

\section{Appendix}
\noindent
{\bf Proof of Theorem \ref{thm2.5}.}
The uniqueness of the solution follows from Corollary \ref{cor2.4}.

To prove the existence, let $x\in D$ be such that $H_D f(x)<  +\infty$ and let  $\sigma := \varepsilon_x  \circ H_D$, it is a probability measure on
$\partial D$.
Let further 
$\mathcal{M} $ be the set of all  functions $ \varphi  \in L^1_+(\partial D, \sigma) $ such that there exists 
$ g \in \mathcal{B}_+(\partial D)$ with  $ \; H_D \varphi \mathop{\longrightarrow}\limits^{k} \varphi$, %on $[ k < \infty ] $,
where $k := H_D g$  and $k(x) < +\infty$.   %\in L^1(D, \lambda)$.
Note that by Theorem 3.7 from \cite{Port-Stone}, page 106,  
it follows that $h$ and  $k$  are  (real-valued) harmonic functions on $D$.
Since for every $\varphi\in C(\partial D)$ the  classical Dirichlet problem has a solution and  %by Theorem \ref{thm3.8}
the solution is precisely $H_D\varphi$, 
from Remark \ref{rem4.1} we get  $C_+(\partial D) \subset \mathcal{M}$. 
We claim that it is sufficient to prove that:
\begin{equation} \label{cond3.1}
\mbox{  if } (\varphi_n)_n \subset \mathcal{M},  \varphi_n \nearrow \varphi \in L^1(\partial D, \sigma), \mbox{ then } \varphi \in \mathcal{M}.
\end{equation}
Indeed, if $(\ref{cond3.1})$ holds then we apply  the monotone class theorem %(cf. (A.2)  in Appendix)    
for 
$b\cb(\partial D)\cap \mathcal{M}$  and $C_+(\partial D)$ as the multiplicative class.
It follows that $b\cb(\partial D)\subset \cm$. 
Let now $\varphi \in L^1_+(\partial D, \sigma)$. 
From the above considerations the sequence $(\varphi \wedge n)_n$ lies in $\cm$
and applying again $(\ref{cond3.1})$ we conclude that $\varphi \in \cm$, 
hence $\cm=L^1_+ (\partial D, \sigma)$.  
%Note that since the control function $k$ is hyperharmonic on $D$, it belongs to $L^1(D, \lambda)$, and $D$ is a domain, it follows that the set $[k=\infty]$ is polar.

Further we argue as in the proof of Theorem 4.8 from \cite{Be11}.
To prove $(\ref{cond3.1})$ let   $(\varphi_n)_n \subset \mathcal{M},$ $\varphi_n \nearrow \varphi \in L^1(\partial D, \sigma)$, 
and set 
$h_n := H_D \varphi_n$ and  $h := H_D \varphi$.
Then $ h_n \nearrow h$ and $h(x)< +\infty$. % \in L^1(D, \lambda)$.
By hypothesis $h_n \mathop{\longrightarrow}\limits^{k_n} \varphi_n$ %on $[k_n < \infty ]$ 
for all $n$.
We may assume that  $k_n (x)  = 1$ for all $n$ and
define  $ k_o:= {\sum}_{n} \frac{k_n}{2^n}$.
It follows that $h_n \mathop{\longrightarrow}\limits^{k_o} \varphi_n$ %on $[k_o< \infty]$ 
for all $n$.
Let
$l := \mathop{\sum}\limits_{n \geq 1} n(h_{n+1} - h_n) = \mathop{\sum}\limits_{n \geq 1} (h - h_n).$
From $h_n \nearrow h $ we have $ h_n(x)  \nearrow h(x)  < +\infty$.
Passing to a subsequence  we may assume that 
$\sum_{n} (h(x) - h_n(x)) < \infty$. 
Consequently,  we get $l (x)< +\infty$ %\in L^1_+(D, \lambda)$ 
and $l = H_D g$ with   $g \in \mathcal{B}_+(\partial D)$.
By  Proposition 1.7 from \cite{Cornea98}
we conclude  that $ h \mathop{\longrightarrow}\limits^{k_o + l} \varphi$  %on $[k_o + l < \infty]$, 
so, $\varphi \in \mathcal{M}$
and therefore $(\ref{cond3.1})$ holds, completing the proof.  $\hfill\square$

\vspace{2mm}

%\pagebreak

\noindent
{\bf Acknowledgments.} 
This work was supported by a grant of the Ministry of Research, Innovation and Digitization, CNCS - UEFISCDI,
project number PN-III-P4-PCE-2021-0921, within PNCDI III. 
The authors acknowledge enlightening discussions with Gheorghe Bucur and Iulian C\^ impean during the preparation of this paper.

%\noindent
%{\bf Competing interests}  There are no competing interests related to this article.

\end{document}